\documentclass[ECP]{ejpecp} 

\usepackage[utf8]{inputenc}

\usepackage{bbm}
\usepackage[english]{babel}
\usepackage[rightcaption]{sidecap}
\sidecaptionvpos{figure}{c}

\SHORTTITLE{The glassy phase of the complex temperature BBM} 

\TITLE{The glassy phase of the complex\\
branching Brownian motion
energy model\thanks{L.H.~is supported by the German Research Foundation in the Bonn
International Graduate School in Mathematics (BIGS), and the Collaborative
Research Center 1060 ``The Mathematics of Emergent Effects''. The authors thank
their home institutions for hospitality.
}} 

\AUTHORS{%
  Lisa Hartung\footnote{Institut f\"ur Angewandte Mathematik, Rheinische Friedrich-Wilhelms-Universit\"at,  Bonn, Germany.  \EMAIL{lhartung@uni-bonn.de}}
  \and 
  Anton Klimovsky\footnote{Fakult\"at f\"ur Mathematik, Universit\"at Duisburg-Essen, Essen, Germany. \EMAIL{anton.klymovskiy@uni-due.de}; \url{http://www.aklimovsky.net}}}


\AMSSUBJ{60J80; 60G70; 60F05; 60K35; 82B44} 

\SUBMITTED{April 20, 2015} 
\ACCEPTED{October 20, 2015} 

\VOLUME{0}
\YEAR{2015}
\PAPERNUM{0}
\DOI{vVOL-PID}

\ABSTRACT{We identify the fluctuations of the partition function for a class of
random energy models, where the energies are given by the positions of the
particles of the complex-valued branching Brownian motion (BBM). Specifically,
we provide the weak limit theorems for the partition function in the so-called
``glassy phase'' -- the regime of parameters, where the behaviour of the
partition function is governed by the extrema of BBM. We allow for arbitrary
correlations between the real and imaginary parts of the energies. This extends
the recent result of Madaule, Rhodes and Vargas~\cite{MRV13}, where the
uncorrelated case was treated. In particular, our result covers the case of the
real-valued BBM energy model at complex temperatures.}

\def\paragraph#1{\noindent \textbf{#1}}

\numberwithin{equation}{section}

\def\ddd{\mathrm{d}}
\def\eee{\mathrm{e}}

\def\<{\langle}
\def\>{\rangle}

\def\a{\alpha}
\def\b{\beta}
\def\e{\epsilon}

\def\g{\gamma}

\def\l{\lambda}

\def\s{\sigma}
\def\t{\tau}

\def\D{\Delta}

\def\R{{\Bbb R}}
\def\N{{\Bbb N}}
\def\P{{\Bbb P}}
\def\Z{{\Bbb Z}}

\def\C{{\Bbb C}}
\def\E{{\Bbb E}}

\let\cal=\mathcal

\def\DD{{\cal D}}
\def\EE{{\cal E}}

\def\GG{{\cal G}}

\def\SS{{\cal S}}
\def\TT{{\cal T}}

\def\UU{{\cal U}}

\def\XX{{\cal X}}

\def \b {{\beta}}
\def \s {{\sigma}}

\def \D {{\Delta}}
\def \t {{\tau}}

\def \g {{\gamma}}
\def \l {{\lambda}}
\def \d {{\delta}}
\def \a {{\alpha}}

\def \ba {\begin{array}}
\def \ea {\end{array}}

\newcommand{\be}{\begin{equation}}
\newcommand{\ee}{\end{equation}}

\newcommand{\bea}{\begin{eqnarray}}
\newcommand{\eea}{\end{eqnarray}}
\def\TH(#1){\label{#1}}\def\thv(#1){\ref{#1}}
\def\Eq(#1){\label{#1}}\def\eqv(#1){(\ref{#1})}

\def\cov{\hbox{\rm Cov}}

\def\sfrac#1#2{{\textstyle{#1\over #2}}}

\def \1{\mathbbm{1}}
\def\wt {\widetilde}

\def\eee{\hbox{\rm e}}

\allowdisplaybreaks

\begin{document}

\section{Introduction}
\label{sec:intro}

Phase transitions arise via an analyticity breaking of the logarithm of the
partition function (see, e.g., Ruelle~\cite{Ruelle1969}). To analyse this
phenomenon, the study of partition functions at \textit{complex temperatures} is
of a key interest, as was observed by Lee and Yang~\cite{YL52, LY52}. Another
motivation to study complex-valued Hamiltonians comes from quantum physics.
There, partition functions with complex energies emerge naturally, e.g., from
the Schrödinger equation via ``imaginary time'' Feynman's path integrals.

It is believed that large classes of models of disordered systems fall in the
same universality class and, in particular, share the same shape of the phase
diagram. Random energy models were proven to be useful in exploring universality
classes in mean-field disordered systems, see, e.g., 
Bovier~\cite{BovStatMech}, Panchenko~\cite{PanchenkoBook2013} and
Kistler~\cite{Kistler2014}. A number of random energy models with complex
energies has been considered in the literature. One of the simplest such models
(in terms of the correlation structure of the energies) is the so called {\it
Random Energy Model} (REM). For this model, the analyticity of the log-partition
function was studied in the seminal work by Derrida~\cite{DerridaComplexREM1991}
and later by Koukiou~\cite{koukiou}. The full phase diagram of this model at
complex temperatures including the fluctuations and zeros of the partition
function were identified by Kabluchko and one of us in \cite{KaKli14}. In
particular, the case of arbitrary correlations between the imaginary and real
parts of the energies was considered in \cite{KaKli14}. The same authors
answered in \cite{KaKli14_G} similar questions about the \textit{Generalized
Random Energy model} (GREM) -- a model with hierarchical correlations -- and
obtained the full phase diagram. In the complex GREM, the phase diagram turned
out to have a much richer structure than that of the complex REM. This sheds
some light on the phase diagrams of the models beyond the complex REM
universality class.

It is known that models with \textit{logarithmic correlations} between the
energies are at the borderline of the REM universality class. In particular,
they are expected to have the same phase diagram. This has been shown for
directed polymers on a tree  with complex-valued energies by Derrida, Evans, and
Speer~\cite{DerridaEvansSpeer1993}, and for a model of complex multiplicative
cascades by Barral, Jin, and Mandelbrot~\cite{BarralJinMandelbrot2010}. Lacoin,
Rhodes, and Vargas~\cite{LRV14} analysed the phase diagram for complex
\textit{Gaussian multiplicative chaos} -- a model with logarithmic correlations
between the energies on a Euclidean space. There, only the case without
correlations between the imaginary and real parts of the energy was treated. It
turned out that the phase diagram coincides with the REM one, see
Figure~\ref{fig-rem-phase-diagram}.

\begin{SCfigure}
\centering \includegraphics[width=0.6\textwidth]{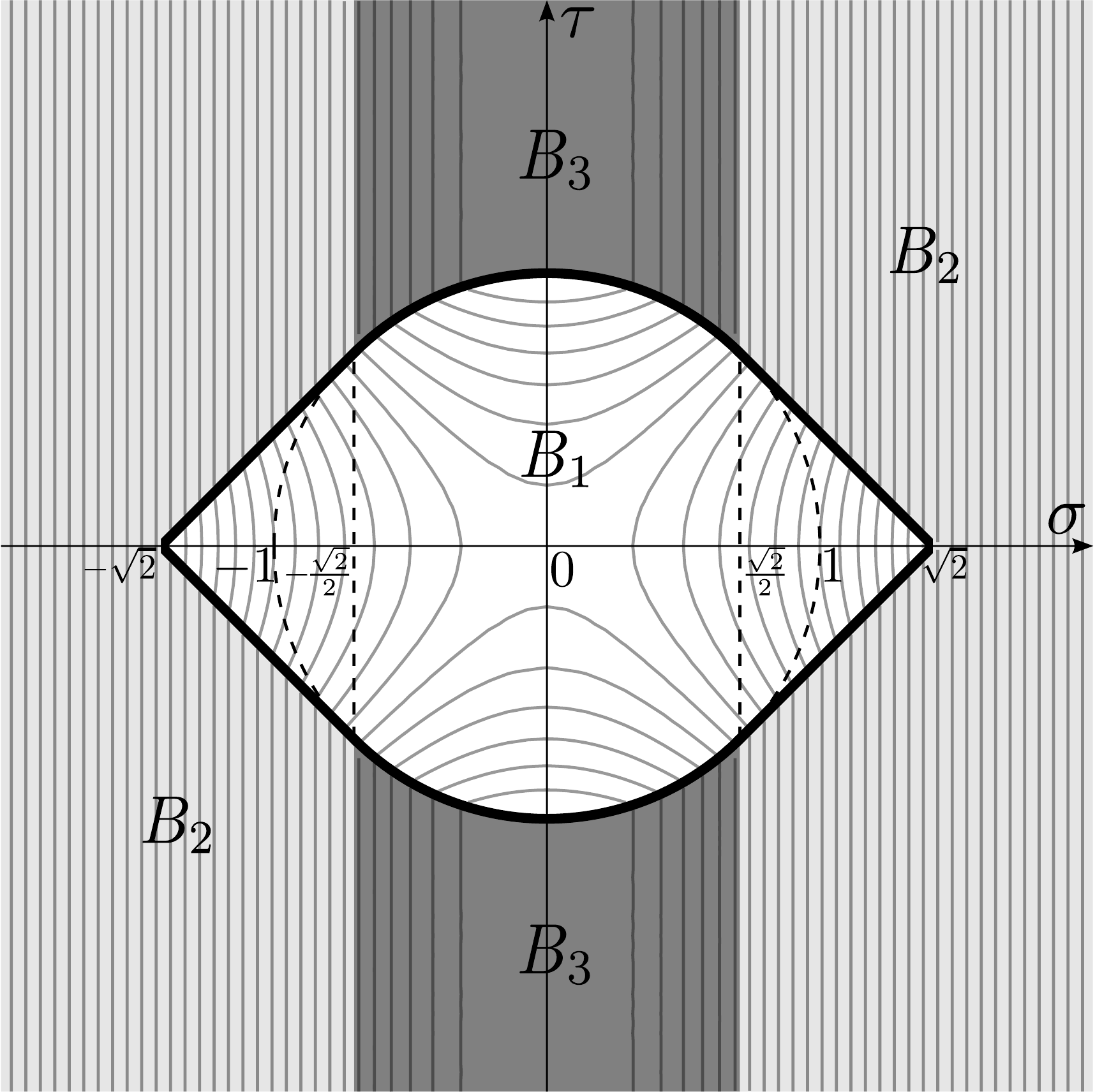}\caption{Phase diagram of the REM (and conjecturally of the BBM energy model).
The grey curves are the level lines of the limiting log-partition function,
cf.~\eqref{eq:limiting-log-partition-function}. 
This paper mainly deals with phase $B_2$.} \label{fig-rem-phase-diagram}
\end{SCfigure}

In~\cite{LRV14}, the analysis of the so-called ``glassy'' phase $B_2$, see
Figure~\ref{fig-rem-phase-diagram},  was left open. In this phase, the partition
function is dominated by the extreme values of the energies. Phase $B_2$ was
analysed by Madaule, Rhodes, and Vargas~\cite{MRV13} in a continuous model with
logarithmic correlations on a tree -- the complex \textit{BBM energy model}, but
again only when the imaginary and real parts of the energies are uncorrelated.
In this model, a deeper understanding of phase $B_2$ is possible due to recent
progress in the analysis of the extremal process of BBM by
A\"\i{}d\'ekon, Berestycki, Brunet, and Shi~\cite{ABBS} and Arguin, Bovier, and
Kistler~\cite{ABK_E}. Madaule, Rhodes, and Vargas~\cite{Madaule2015Continuity},
have recently analysed the behaviour of the partition function on the boundary
between phases $B_1$ and $B_2$ (see Figure~\ref{fig-rem-phase-diagram}).

In this article, we extend the result of \cite{MRV13}. Specifically, we prove
the weak convergence of the (rescaled) partition function of the complex BBM
energy model in phase $B_2$ to a non-trivial distribution. We allow for
arbitrary correlations between the real and imaginary parts of the energy. In
particular, this covers the complex temperature case, in which the real and
imaginary parts of the random energies have maximal correlation (i.e., they are
a.s.~equal). This case is especially relevant for the Lee-Yang program.

\subsection{Branching Brownian motion.} 

Before stating our results, let us briefly recall the construction of a BBM.
Consider a canonical continuous branching process: a \textit{continuous time
Galton-Watson} (GW) process \cite{AN}. It starts with a single particle at time
zero. After an exponential time of parameter one, this particle splits into $k
\in \Z_+$ particles according to some probability distribution $(p_k)_{k \geq
0}$ on $\Z_+$. Then, each of the new-born particles splits independently at independent
exponential (parameter $1$) times again according to the same $(p_k)_{k \geq 0}$, and so on. We
assume that $\sum_{k=1}^\infty p_k=1.$\footnote{This implies that $p_0 = 0$, so
none of the particles ever dies.} In addition, we assume that
$\sum_{k=1}^\infty k p_k=2$ (i.e., the expected number of children per particle
equals two)\footnote{The latter assumption is just a matter of normalization.
Any expected number of children greater than $1$ (= the supercritical regime) is
allowed and the results of this paper remain valid with appropriate
modifications of constants.}. Finally, we assume that $K :=
\sum_{k=1}^\infty k(k-1)p_k<\infty$ (finite second
moment)\footnote{\label{note:3}Under the stated conditions, the convergence of
the extremal process of BBM, on which we rely, is proven in \cite{ABK_E}. For
the case of branching random walk, using truncation techniques, Madaule
\cite{Madaule2011} has shown the same under conditions that would in the
Gaussian case imply finiteness of $\sum_kp_kk (\ln k)^3$.  This could probably
be carried over to BBM. It is not clear whether the result holds under the
Kesten-Stigum  condition $\sum_kp_kk \ln k<\infty$. For a discussion on these
issues, we refer to the lecture notes by Shi~\cite{Shi2015}. In the present
paper, we are not concerned with improving the conditions on the offspring
distribution.}. At time $t = 0$, the GW process starts with just one particle.

For given $t \geq 0$, we label the particles of the process as $i_1(t),\dots,
i_{n(t)}(t)$, where $n(t)$ is the total number of particles at time $t$. Note
that under the above assumptions, we have $\E\left[ n(t) \right]=\eee^t$. For $s
\leq t$, we denote by $i_k(s,t)$ the unique ancestor of particle $i_k(t)$ at
time $s$. In general, there will be several indices $k, l$ such that $i_k(s,t) =
i_l(s,t)$. For $s, r \leq t$, define the time of the most recent common ancestor
of particles $i_k(r,t)$ and $i_l(s,t)$ as
\begin{align}
\label{eq:intro:gw-overlap}
d(i_k(r,t), i_l(s,t))
:= \sup \{u \leq s \wedge r \colon i_k(u,t) = i_l(u,t)\}.
\end{align} 
For $t \geq 0$, the collection of all ancestors naturally induces the random tree
\begin{align}
\label{eq:gw-tree}
\mathbb{T}_t := \{i_k(s,t) \colon 0 \leq s \leq t, 1\leq k \leq n(t) \}
\end{align}
called the \textit{GW tree up to time $t$}. We denote by $\mathcal{F}^{\mathbb{T}_t}$
the $\s$-algebra  generated by the GW process up to time $t$.

In addition to the genealogical structure, the particles get a
\textit{position} in $\R$. Specifically, the first particle starts at the origin
at time zero and performs Brownian motion until the first time when the GW
process branches. After branching, each new-born particle independently performs
Brownian motion (started at the branching location) until their respective next
branching times, and so on. We denote the positions of the $n(t)$ particles at
time $t \geq 0$ by $x_1(t),\dots, x_{n(t)}(t)$ and by $x_1(s,t),\dots,
x_{n(t)}(s,t)$ the positions of their ancestors at time $s \geq 0$.

We define BBM as a family of Gaussian processes, 
\begin{align}
\label{eq:bbm-process}
x_t := \{ x_1(s,t),\dots, x_{n(t)}(s,t) \colon s \leq t \}
\end{align}
indexed by time horizon $t \geq 0$. Note that conditionally on the underlying GW tree these Gaussian processes have
the following covariance
\begin{align}
\E \left[x_k(s,t) x_l(r,t) \mid \mathcal{F}^{\mathbb{T}_t} \right] = d(i_k(s,t), i_l(r,t)), \quad  s, r \in [0,t], \quad k,l \leq n(t)
.
\end{align}
Bramson~\cite{B_C,B_M} showed that
\begin{align}
\label{eq:extremal-scaling}
m(t):=\sqrt 2 t-\frac{3}{2\sqrt 2}\log t
\end{align}
is the order of the maximal position among all BBM particles alive at large time $t$, i.e.,
\begin{align}
\label{eq:shifted-gumbel}
\lim_{t \uparrow \infty} \P \left\{ \max_{k \leq n(t)} x_k(t) - m(t) \leq y \right\} = \E \left[ \eee^{-C Z \eee^{-\sqrt{2} y}} \right], \quad y \in \R
,
\end{align}
where $C > 0$ is a constant and  $Z$ is the a.s.~limit of the so-called \textit{derivative martingale}:
\begin{align}
Z := \lim_{t \uparrow \infty} \sum_{k=1}^{n(t)} (\sqrt{2} t - x_k(t)) \eee^{-\sqrt{2}(\sqrt{2}t-x_k(t))}, \quad \text{a.s.}
\end{align}
In \cite{ABBS,ABK_E}, as $t \uparrow \infty$, the non-trivial limiting point
process of the (shifted by $m(t)$) particles of BBM was identified.
Specifically, it was shown that the point process,
\begin{align}\Eq(extremal.1)
\EE_t := \sum_{k=1}^{n(t)}\d_{x_k(t)-m(t)}, \quad t \in \R_+
\end{align}
converges in law as $t\uparrow \infty$ to the point process
\begin{align}\Eq(extremal.2)
\EE := \sum_{k,l}\d_{\eta_k+\D^{(k)}_l},
\end{align}
where:
\begin{itemize}

\item[(a)] $\{\eta_k \}_{k \in \N} \subset \R$ are the atoms of a Cox process
with \textit{random intensity measure} $CZ \eee^{-\sqrt 2 y}\ddd y$, where $C$
and $Z$ are the same as in \eqref{eq:shifted-gumbel}.

\item[(b)] $\{\D^{(k)}_l \}_{l \in \N} \subset \R$ are the atoms of independent
and identically distributed point processes $\D^{(k)}$, $k  \in \N$ called
\textit{clusters} which are independent copies of the limiting point process
\begin{align}\Eq(extremal.3)
\D := \lim_{t\uparrow \infty}
\sum_{k=1}^{n(t)}\d_{\hat x_k(t)-\max_{l\leq n(t)}\hat x_l(t)}
\end{align}
with $\hat x(t)$ being BBM $x(t)$ conditioned on $\max_{k\leq n(t)} x_k(t)\geq \sqrt 2 t$.

\end{itemize}

\subsection{Branching Brownian motion energy model at complex temperatures  with arbitrary correlations}

Let  $\rho \in [-1,1]$. For any $t \in \R_+$, let $X(t) := (x_k(t))_{k\leq
n(t)}$ and $Y(t) :=(y_k(t))_{k\leq n(t)} $ be two BBMs with the same underlying
GW tree such that, for $k \leq n(t)$,
\begin{align}
\cov(x_k(t),y_k(t))=|\rho| t
.
\end{align}
Then,  
\begin{align}\Eq(cor.1)
Y(t)\overset{\mathrm{D}}{=} \rho X(t)  +\sqrt{1-\rho^2} Z(t)
,
\end{align}
where ``$\overset{\mathrm{D}}{=}$'' denotes equality in distribution and $Z(t) :=
(z_i(t))_{i\leq n(t)}$ is a branching Brownian motion with the same underlying
GW process which is independent from $X(t)$. 
Representation $\eqv(cor.1)$ allows us to handle arbitrary correlations by
decomposing the process $Y$ into a part independent from $X$ and a fully
correlated one.

We define the partition function for the complex
BBM energy model with correlation $\rho$ at inverse temperature $\b := \s+i\t\in\C $ by
\begin{align}
\Eq(real.1)
\wt \XX_{\b,\rho}(t) := \sum_{k=1}^{n(t)}\eee^{\s x_k(t)+i\t y_k(t)}
.
\end{align}

\subsection{Main results} 

Let us specify the three phases depicted on Figure~\ref{fig-rem-phase-diagram} analytically:
\begin{equation}
\begin{aligned}
B_{1}
:=
\C \setminus \overline{B_2\cup B_3},
\quad
&
B_{2}
:=
\{ \sigma + i \tau \in \C \colon 2\sigma^2 >  1, |\sigma|+|\tau| > \sqrt {2}\}
,
\\
&
B_{3}
:=
\{ \sigma + i \tau  \in \C \colon 2\sigma^2 < 1, \sigma^2+\tau^2> 1\}.
\end{aligned}
\end{equation}

In this paper, we focus on the \emph{glassy phase} $B_2$. We start with the convergence
of the partition function in the case of the real BBM energy model at complex
temperatures. We say that a complex-valued r.v.\ $Y$ is \textit{isotropic
$\alpha$-stable} if there exists $c \in \R_+$ and $\alpha \in (0,2]$ such that
\begin{align}
\E [\eee^{i \mathrm{Re} (\bar{z} Y)}] = \eee^{-c |z|^\alpha}, \quad \text{for all } z \in \C
.
\end{align}
Recall the notation from \eqref{extremal.2}.
\begin{theorem}[Partition function fluctuations for $|\rho| = 1$]\TH(thm.real)
For $\b=\s+i\t \in B_2$, the rescaled partition function
$
\XX_{\b,1}(t):= \eee^{-\b m(t)} \wt \XX_{\b,1}(t)
$
converges in law to the r.v.
\begin{align}\Eq(real.2)
\XX_{\b,1}:=\sum_{k,l \geq 1} \eee^{\b \left(\eta_k+\D_l^{(k)}\right)}, \quad \text{as $t\uparrow \infty$.}
\end{align}
\end{theorem}
\begin{theorem}[Partition function fluctuations for $|\rho| \in (0,1)$]\TH(thm.real2)

For $\b=\s+i\t \in B_2$ and $|\rho|\in (0,1)$, the
rescaled partition function
$
\XX_{\b,\rho}(t):= \eee^{-\s m(t)}\wt \XX_{\b,\rho}(t)
$
converges in law to the r.v. $ \XX_{\b,\rho}$, as $t\uparrow \infty$.
Conditionally on $Z$, $ \XX_{\b,\rho}$ is a complex isotropic $\sqrt
2/\s$-stable r.v.
\end{theorem}
\begin{remark} For $\rho = 0$, Theorem~\ref{thm.real2} was proven in \cite{MRV13}. 
Our proof uses a representation of correlated real and imaginary parts in terms
of independent BBM's. As in \cite{MRV13}, we control second moments. However,
the way we do this is different and simpler then the method used in that paper,
which relies on decomposing the paths of the BBM particles according to the time
and location of the minimal position along the given path. Our approach uses
instead the upper envelope for ancenstral paths that was obtained in
\cite{ABK_G}.
\end{remark}

\begin{remark}
Note that the fluctuations of the partition function in the complex BBM energy
model (cf., Theorems~\ref{thm.real}, \ref{thm.real2}) are governed by the
extremal process $\mathcal{E}$. Thus, the fluctuations are different from the
ones in the complex REM \cite[Theorems~2.8, 2.20]{KaKli14} which are governed by
a Poisson point process. Despite the differences in fluctuations, we conjecture
that in the limit as $t \uparrow \infty$ the \textit{log-partition function}
\begin{align}
\label{eq:free-energy}
p_t(\beta) :=  \frac{1}{t} \log |\wt \XX_{\b,\rho}(t)|,
\quad t \in \R_+, \quad \beta \in \C
\end{align}
of the complex BBM energy model is the same as in the complex REM.
\end{remark}

\begin{conjecture}[Phase diagram]
\label{conj:phase-diagram}
For any $\rho \in [-1,1]$, the complex BBM energy model has the same free energy
and the phase diagram (cf., \textup{Figure~\ref{fig-rem-phase-diagram}}) as the
complex REM, i.e.,
\begin{align}
\label{eq:limiting-log-partition-function}
\lim_{t \uparrow \infty} p_t(\beta) 
=:
p(\beta) =
\begin{cases}
1 + \frac{1}{2}(\sigma^2 - \tau^2), & \beta \in \overline{B_1},
\\
\sqrt{2}|\sigma|, & \beta \in \overline{B_2},
\\
\frac{1}{2}+\sigma^2, & \beta \in \overline{B_3},
\end{cases}
\end{align}
and the convergence in \eqref{eq:limiting-log-partition-function} holds in
probability and in $L^1$.
\end{conjecture}

\begin{remark}
Convergence in probability for $\beta \in B_2$ in
\eqref{eq:limiting-log-partition-function} follows from
Theorems~\ref{thm.real} and \ref{thm.real2} by \cite[Lemma~3.9~(1)]{KaKli14}. 
The remaining Parts~$B_1$ and $B_3$ of Conjecture~\ref{conj:phase-diagram} are supported by
results for similar models, e.g.,~\cite{DerridaEvansSpeer1993,
BarralJinMandelbrot2010, LRV14, KaKli14, KaKli14_G} and by the following intuition.

For $\beta \in B_1$, $\widetilde{\XX}_{\b,\rho}(t)/\E[\widetilde{\XX}_{\b,\rho}(t)]$ is an
$L^1$-convergent complex-valued martingale (as $t \to \infty$) with expectation
$1$ and a simple computation shows that
\begin{align}
\label{eq:expected-partition-function}
\vert \E[\widetilde{\XX}_{\b,\rho}(t)] \vert = \exp\left(t + \frac{1}{2} t(\sigma^2 - \tau^2)\right)
.
\end{align}
See Appendix~\ref{sec:martingale-convergence} for the $L^2$-martingale
convergence in the domain $|\beta| < 1$. 

For $\beta \in B_3$, the variance of the partition function of the REM
with $\eee^t$ independent particles equals
\begin{align}
 \eee^t \left(\E [\exp( 2 \s x_1(t))] - \exp\left(\frac{1}{2}t (\sigma^2 - \tau^2) \right)\right) 
\underset{t \uparrow \infty}{\sim}
\exp\left(  t + 2\sigma^2 t \right)
,
\end{align}
cf.~\cite{KaKli14}. Therefore, as $t \uparrow \infty$, the standard deviation
has a greater order of magnitude than the expectation 
\eqref{eq:expected-partition-function}.
So, in view of the central limit theorem, it is plausible that
\begin{align}
\label{eq:fluctuation-normalization}
\widetilde{\XX}_{\b,\rho}(t)/ \exp\left( \frac{1}{2} t + \sigma^2 t \right)
\end{align}
converges as $t \uparrow \infty$ in distribution. However, due to correlations
between the particle positions of BBM, the limiting distribution in
\eqref{eq:fluctuation-normalization} need not be Gaussian,
cf.~\cite[Theorems~4.2 and 6.6]{LRV14} and \cite[Eq.~(2.11)]{KaKli14}.
\end{remark}

\paragraph{Organization of the rest of the paper.} The proofs of
Theorems~\thv(thm.real) and \thv(thm.real2) consist of two main steps. First, we
show that only the extremal particles can contribute to the partition function
in the limit as $t \uparrow \infty$ (cf., Proposition~\ref{Prop.nocont} and its
proof in Section~\ref{sec:proof-of-proposition}). Second, we use the continuous
mapping theorem to deduce Theorems~\thv(thm.real) and \thv(thm.real2) from the
behaviour of the extremal process. This is done in
Section~\ref{sec:partition-function}.

\section{Convergence of the partition function} \label{sec:partition-function}

First, we state that in the glassy phase $B_2$ only the extremal particles can
contribute to the limit of the partition function as $t$ tends to infinity.

\begin{proposition}\TH(Prop.nocont)
If $|\rho| \in (0,1]$ and $\beta \in B_2$, then, for all $\d,\e>0$, there exists
$A_0 > 0$ such that, for all $A>A_0$ and all $t$ sufficiently large,
\begin{align}\Eq(rho.1)
\P\Biggl\{\Bigl\vert \sum_{k=1}^{n(t)}\eee^{\s(x_k(t)-m(t))+i\t y_k(t)}\1_{\{x_k(t)-m(t)<-A\}}\Bigr\vert>\d\Biggr\}<\e.
\end{align}
\end{proposition}
The proof of Proposition~\thv(Prop.nocont) is postponed until
Section~\ref{sec:proof-of-proposition}. Using Proposition \thv(Prop.nocont)
together with the continuous mapping theorem, we now prove
Theorem~\thv(thm.real).
\begin{proof}[Proof of Theorem~\thv(thm.real)]

Denote by $\mathbb{M}$ the space of locally finite counting measures on
$\overline{\R} := \R \cup \{ +\infty\}$. We endow $\mathbb{M}$ with the 
vague topology. Consider for $A\in \R_+$ the functional
$\Phi_{\b,A}\colon \mathbb{M} \to \R$. This functional maps a locally
finite counting measure $\zeta =\sum_{i\in I}\d_{x_i}$ to
$\Phi_{\b,A}(\zeta):=\sum_{i\in I}\eee^{\b x_i}\1_{\{x_i>-A\}}$, where $I$ is a
countable index set. The set of locally finite measures $\zeta$ on which the
functional $\Phi_{\b,A}$ is not continuous (i.e., $\zeta$ charging $-A$ or
$+\infty$) has zero measure w.r.t.~the law of $\mathcal{E}$. Hence, by the
continuous mapping theorem, it follows that $\Phi_{\b,A}(\EE_t)$ converges in
law to $\Phi_{\b,A}(\EE)$, which is equal to
\begin{align}\Eq(real.3) 
\sum_{k,l\geq 1} \eee^{\b \left(\eta_k+\D_l^{(k)}\right)}\1_{\{\eta_k+\D_l^{(k)}\geq -A\}}.
\end{align}
Note that by Proposition \thv(Prop.nocont), for all $\e>0$ and $\d>0$, there
exists $A_0$ such that, for all $A>A_0$ and all $t$ sufficiently large,
\begin{align}
\P\left\{\left\vert\XX_{\b,1}(t)-\Phi_{\b,A}(\EE_t)\right\vert>\d \right\}<\e.
\end{align} 
Hence, by Slutsky's Theorem (see, e.g., \cite[Theorem~13.18]{Klenke2008}), $\XX_{\b,1}(t)$ converges in law to 
\begin{align}\Eq(real.4)
\lim_{A\uparrow \infty}\sum_{k,l\geq 1} e^{\b \left(\eta_k+\D_l^{(k)}\right)}\1_{\{\eta_k+\D_l^{(k)}\geq -A\}}
\end{align}
which is equal to $\XX_{\b,1}$.
\end{proof}
We now prove  Theorem~\thv(thm.real2).
\begin{proof}[Proof of Theorem~\thv(thm.real2)]
Using Representation \eqv(cor.1), we have that $\XX_{\b,\rho}(t)$ is in
distribution equal to
\begin{align}\Eq(real.10)
\sum_{k=1}^{n(t)}\eee^{(\s+i\rho \t)(x_k-m(t))+i \sqrt{1-\rho^2}\t z_k(t)-i\rho\t m(t)},
\end{align}
where $(z_k(t),k\leq n(t))$ are the particles from a BBM
that is independent from $X(t)$ (but with respect to the same GW
tree). If $|\rho|\neq 1$, then by \cite[see Lemma~3.2 and the
subsequent discussion before Eq.~(3.7) therein]{MRV13} we get that
\begin{align}\Eq(real.11)
G(t) := \sum_{k=1}^{n(t)}\d_{(x_k(t)-m(t),\exp (i\sqrt{1-\rho^2}\t z_k(t)-i\rho\t m(t)))} 
\end{align}
converges weakly as $t \uparrow \infty$ to
\begin{align}\Eq(real.12)
\GG := \sum_{k,l\geq 1} \d_{(p_k+\D_l^{(k)},U^{(k)}\wt W_l^{(k)})}, 
\end{align}
where $\left(U^{(k)}\right)_{k\geq 1}$ are i.i.d.~uniformly distributed on the
unit circle and $\wt W_l^{(k)}$ are the atoms of a point process on the unit
circle. The description of $\wt W^{(k)}$ could be made more explicit using the
description of the cluster process $\Delta$ obtained in \cite[Theorem~2.3]{ABBS}
that encodes the genealogical structure of $\D$.

Denote by $\wt{\mathbb{M}}$ the space of locally finite counting measures on
$\overline{\R} \times \{z \in \C \colon |z| = 1 \}$. We endow $\wt{\mathbb{M}}$
with the (Polish) topology of vague convergence. For $A\in \R_+$, consider the
functional $\wt\Phi _{\b,A} \colon \wt{\mathbb{M}} \to \C $ that maps a locally
finite counting measure $\tilde \zeta =\sum_{k\in I}\d_{(x_k,z_k)}$ to $\wt\Phi
_{\b,A}(\zeta):=\sum_{k\in I}\eee^{\b x_k}z_k\1_{\{x_k>-A\}}$, where $I$ is a
countable index set. The set of locally finite measures $\zeta$ on which the
functional $\Phi_{\b,A}$ is not continuous (i.e., $\wt\zeta$ charging
$(-A,\cdot)$ or $(+\infty, \cdot)$) has zero measure w.r.t.~the law of
$\mathcal{G}$. Hence, by the continuous mapping theorem, it follows that
$\wt\Phi_{\s+i\rho \t,A}(\GG_t)$ converges in law to $\wt\Phi_{\s+i\rho
\t,A}(\GG)$, which is equal to
\begin{align}\Eq(real.13)
\sum_{k,l\geq 1} \eee^{(\s+i \rho \t)\left(\eta_k+\D_l^{(k)}\right)}U^{(k)}\wt W_l^{(k)}\1_{\{\eta_i+\D_l^{(k)}\geq -A\}}
.
\end{align}  
Since $\eee^{(i \rho \t)\left(\eta_k+\D_l^{(k)}\right)} U^{(k)}$ is also uniformly distributed on the unit circle, 
 \eqv(real.13) 
is  equal  in distribution to
\begin{align}\Eq(real.14)
\sum_{k,l\geq 1} \eee^{\s\left(\eta_k+\D_l^{(k)}\right)}U^{(k)} \wt W_l^{(k)}\1_{\{\eta_i+\D_l^{(k)}\geq -A\}}
.
\end{align}
Note that again by Proposition~\thv(Prop.nocont), for all $\e>0$ and $\d>0$,
there exists $A_0$ such that, for all $A>A_0$ and all $t$ sufficiently large,
\begin{align}
\P\left\{\left\vert\XX_{\b,\rho}(t)-\wt\Phi_{\s+i\rho\t,A}(\GG_t)\right\vert>\d\right\}<\e.
\end{align} 
Hence, by Slutsky's theorem (see, e.g., \cite[Theorem~13.18]{Klenke2008}),
$\XX_{\b,\rho}(t)$ converges in law to
\begin{align}\Eq(real.4a)
\lim_{A\uparrow \infty}\sum_{k,l\geq 1} \eee^{\s\left(\eta_k+\D_l^{(k)}\right)}U^{(k)}\wt W_l^{(k)}\1_{\{\eta_k+\D_l^{(k)}\geq -A\}}
= 
\sum_{k,l\geq 1} \eee^{\s\left(\eta_k+\D_l^{(k)}\right)}U^{(k)}\wt W_l^{(k)}
.
\end{align}
We rewrite \eqv(real.4a) as
\begin{align}
\label{eq:partition-funct}
\sum_{k \geq 1} \eee^{\s\eta_k}U^{(k)} W^{(k)}
,
\end{align}
where $W^{(k)} := \sum_{l} \eee^{\s \D_l^{(k)}} \wt W^{(k)}_l$, $k \geq 1$ are
i.i.d.\ r.v.'s. From \eqref{eq:partition-funct}, it follows that conditionally
on $Z$, the distribution of $ \XX_{\b,\rho}$ is complex isotropic $\sqrt
2/\s$-stable.

\end{proof}

\section{Proof of Proposition~\thv(Prop.nocont)}
\label{sec:proof-of-proposition} 

Due to symmetry, we only prove Proposition~\thv(Prop.nocont) for $\s,\t>0$. In
the proof of Proposition~\thv(Prop.nocont), we distinguish two cases:
\begin{eqnarray}
\text{\textbf{(a)}} \quad \sigma>\sqrt 2; & & \text{\textbf{(b)}} \quad \sqrt{2}/2<\sigma\leq\sqrt 2\text{ and } \s+\t>\sqrt 2.
\end{eqnarray}
\paragraph{Case (a).} In this case, the proof works as in the independent case treated in
\cite[Lemma\;3.5]{MRV13}. For completeness, we also provide the proof in this case.
We use a first moment computation together with the upper bound on the maximal
position of all particles obtained in \cite[Theorem~2.2]{ABK_G}.
\begin{proof}[Proof of \textup{Proposition~\thv(Prop.nocont)} in case \textup{(a)}]
Recall the notation from \eqref{eq:bbm-process}.
By \cite[Theorem~2.2]{ABK_G}, for $0<\g<\frac{1}{2}$, there exists  $r_{\e} > 0$ such
that for all $r>r_{\e}$ and $t>3r$
\begin{align}\Eq(rho.2)
\P\left\{\exists k\leq n(t) \colon x_k(s,t) > U_{t,\g} \mbox{ for some }s\in[r,t-r]\right\}<\frac{\e}{2},
\end{align}
where $U_{t,\g}(s):=\frac{s}{t}m(t)+(s\wedge (t-s))^\g$. Define the following set
on the path space 
\begin{align}\Eq(UU.1)
\UU_{t,r,\g}:=\{x(\cdot) \in C(\R_+, \R)\colon  x(s,t)\leq\frac{s}{t}m(t)+(s\wedge (t-s))^\g, \forall s\in[r,t-r]\}.
\end{align}
By \eqv(rho.2), to show \eqv(rho.1), it suffices to check that, for sufficiently large $A > 0$,
\begin{align}\Eq(rho.3)
\P\Biggl\{\Bigl\vert \sum_{k=1}^{n(t)}\eee^{\s(x_k(t)-m(t))+i\t y_k(t)}\1_{\{x_k(t)-m(t)<-A\}\cap\{ x_k\in\UU_{t,r,\g} \}  }\Bigr\vert>\d\Biggr\}<\e/2.
\end{align}
By Markov's inequality, the probability in \eqv(rho.3) is bounded from above by
\begin{align}\Eq(rho.4)
&\frac{1}{\d}\E\left[\left\vert \sum_{k=1}^{n(t)}\eee^{\s(x_k(t)-m(t))+i\t y_k(t)}\1_{\{x_k(t)-m(t)<-A\}\cap \{x_k\in\UU_{t,r,\g}\} }\right\vert\right]\nonumber\\
&\leq \frac{1}{\d}\E\left[ \sum_{k=1}^{n(t)}\eee^{\s(x_k(t)-m(t))}\1_{\{x_k(t)-m(t)<-A\}\cap \{x_k\in\UU_{t,r,\g}\} }\right]
.
\end{align}
We rewrite the expectation in the r.h.s.~of \eqv(rho.4) as $\sum_{B>A} S(B,t)$, where
\begin{align}\Eq(rho.5)
S(B,t) := \E\left[ \sum_{k=1}^{n(t)}\eee^{\s(x_k(t)-m(t))}\1_{\{x_k(t)-m(t)\in (-B+1,-B]\}\cap \{x_k\in\UU_{t,r,\g} \}}\right].
\end{align}
Next, we manipulate the event 
\begin{align}\Eq(rho.6)
&\{x_k(t)-m(t)\in (-B+1,-B]\}\cap \{x_k\in\UU_{t,r,\g}\}\\
&\subset \{x_k(t)-m(t)\in (-B+1,-B]\}\cap\{\xi(s)\leq  \frac{s}{t}B+(s\wedge(t-s))^\g, \forall s\in[r,t-r]\},\nonumber
\end{align}
where $\xi_k(s):=x_k(s,t)-\frac{s}{t}x_k(t)$ is a Brownian bridge from $0$ to $0$
in time $t$ that is independent from $x_k(t)$. Hence, we can bound $S(B,t)$ from
above by
\begin{align}\Eq(rho.7)
&\E\left[ \sum_{k=1}^{n(t)}\eee^{\s(x_k(t)-m(t))}\1_{\{x_k(t)-m(t)\in (-B+1,-B]\}\cap\{\xi_k(s)\leq  \frac{s}{t}B+(s\wedge(t-s))^\g, \forall s\in[r,t-r]\} }\right]\\\nonumber
&=\eee^t\E\left[\eee^{\s(x(t)-m(t))}\1_{x(t)-m(t)\in (-B+1,-B]}\right]\P\left\{\xi(s)\leq  \frac{s}{t}B+(s\wedge(t-s))^\g, \forall s\in[r,t-r]\right\},
\end{align}
where $x(t)$ is normal distributed with mean $0$ and variance $t$ and
$\xi(\cdot)$ is a Brownian bridge from $0$ to $0$ in time $t$ independent from
$x(t)$. The expectation in the second line of \eqv(rho.7) is equal to
\begin{align}\Eq(rho.8)
\int_{m(t)-B}^{m(t)-B+1}\eee^{\s(x-m(t))}\eee^{-x^2/2t} \frac{\ddd x}{\sqrt{2\pi t}} 
=\eee^{-\s m(t)+\frac{\s^2 t}{2}}\int_{m(t)-B-\s t}^{m(t)-B+1-\s t} \eee^{-w^2/2t}\frac{ \ddd w}{\sqrt{2\pi t}},
\end{align}
where we changed variables $x=w+\s t$ . Since $\s>\sqrt 2$, by the definition of
$m(t)$ it holds that $m(t)-B-\s t<(\sqrt{2}-\s)t<0$, for all $t > 1$. Therefore,
using the standard Gaussian tail bound,
\begin{align}\Eq(Gauss.tail)
\int_{-\infty}^{-x}\eee^{-w^2/2}\frac{\ddd w}{\sqrt{2\pi}}\leq \frac{1}{\sqrt{2\pi}x}\eee^{-x^2/2}, \quad x>0,
\end{align}
we can bound \eqv(rho.8) using $m^2(t) = 2 t - 3 t \log t + \left( 3 \log t / (2
\sqrt{2})\right)^2$ from above by
\begin{align}\Eq(rho.9)
& \sfrac{\sqrt{t}}{\sqrt{2\pi}(B-1+\s t-m(t))}\eee^{-\s m(t)+\frac{\s^2 t}{2}}\eee^{-\left(m(t)-B+1-\s t\right)^2/2t}
\underset{t \uparrow \infty}{\sim} \sfrac{t}{\sqrt{2\pi} (\s-\sqrt 2)}\eee^{-t+(\sqrt 2-\s)(B-1)}
.
\end{align}
Next, we analyse the probability in the r.h.s.~of \eqv(rho.7). We bound it, for
$B<t^\g/3$, from above by
\begin{align}\Eq(rho.10)
\P\left\{\xi(s)\leq  2(s\wedge(t-s))^\g, \forall s\in[r\lor B^{1/\g}
,(t-B^{1/\g})\wedge (t-r)]\right\}.
\end{align}
By the proof of \cite[Theorem~2.3, see (5.55)]{ABK_G}, for all $r$ large enough,
probability \eqv(rho.10) is bounded from above by
\begin{align}\Eq(rho.11)
\P\left\{\xi(s)\leq  0, \forall s\in[r\lor B^{1/\g},(t-B^{1/\g})\wedge (t-r)]\right\}(1+\e)
\leq  \frac{2 (B^{1/\g}\wedge r)}{t-2(B^{1/\g}\wedge r)} (1+\e),
\end{align}
where in the last step we used \cite[Lemma~3.4]{ABK_G}. Plugging the estimates
from \eqv(rho.9) and \eqv(rho.11) into \eqv(rho.7), we get
\begin{align}\Eq(rho.12)
S(B,t)\leq \left( \frac{2(B^{1/\g}\lor r)}{t-2(B^{1/\g}\lor r)} (1+\e)\1_{\{B>t^\g/3\}}+\1_{\{B\leq t^\g/3\}}\right)\frac{t\eee^{(\sqrt 2-\s)(B-1)}}{\sqrt{2\pi} (\s-\sqrt 2)}(1+o(1))
.
\end{align}
Note that in \eqref{rho.12} and below $o(1)$ denotes a $t$-dependent non-random
quantity with
\begin{align}
o(1) \underset{t \uparrow \infty}{\longrightarrow} 0.
\end{align}
From \eqref{rho.12} follows that $\lim_{t\uparrow \infty} \sum_{B>t/3}S(B,t)=0$ and
\begin{align}\Eq(rho.13)
\sum_{B=A+1}^{t^\g/3}S(B,t)\leq
\sum_{B=A+1}^{t^\g/3} \frac{2t(B^{1/\g}\lor r)\eee^{(\sqrt 2-\s)(B-1)}}{\sqrt{2\pi} (\s-\sqrt 2)(t-2(B^{1/\g}\lor r))} 
(1+\e)
,
\end{align}
which can be made smaller than $\e/2$ by taking $A$ large enough since
$\sqrt{B^{1/\g}\wedge r}\eee^{(\sqrt 2-\s)(B-1)}$ is summable in $B$ (because
$\sqrt 2-\s<0$). This concludes the proof of Theorem~\thv(Prop.nocont) in
case (a).
\end{proof}

\paragraph{Case (b).} In this case, the analysis is somewhat more intricate and we
have to employ the imaginary part of the energy. 

\paragraph{Short outline of the proof.} To prove \eqv(rho.1), we first
apply the Chebyshev inequality to the absolute value of the truncated partition
function. Then, we compute the second moment which arises in the Chebyshev
inequality. Along the way, we first use Representation~\eqv(cor.1) and compute
the expectation w.r.t.\ $z(t)$ conditionally on $\mathcal{F}^{\mathbb{T}_t}$,
see \eqv(rho.15). Starting from \eqv(rho.16), we use the so-called upper
envelope for the given path of $x(t)$  (see \cite[Theorem~2.2]{ABK_G}) to
control the expectation w.r.t.\ $x(t)$. Technically, we have to distinguish
between three regimes for the time of the most recent common ancestor
$q_{k,l}=d(x_k(t),x_l(t))$. The corresponding terms are controlled separately
starting from Eq.~\eqv(rho.51).\footnote{Note that this approach to control the second
moment differs from the one used in \cite{MRV13}. The latter one relies on
decomposing the paths of the BBM particles according to the time and location of
the minimal position along the given path.}

\begin{proof}[Proof of Proposition~\thv(Prop.nocont) in case \textup{(b)}]
We proceed as in case (a) until \eqv(rho.3). This time, using Chebyshev's
inequality, we bound the probability in \eqv(rho.3) by
\begin{align}\Eq(rho.14)
\frac{1}{\d^2}\E\Biggl[\Bigl\vert \sum_{k=1}^{n(t)}\eee^{\s(x_k(t)-m(t))+i\t y_k(t)}\1_{\{x_k(t)-m(t)<-A\}\cap\{ x_k\in\UU_{t,r,\g} \}  }\Bigr\vert^2\Biggr],
\end{align}
We introduce the shorthand notation
$ \widetilde{x}_k(t) := x_k(t)-m(t)$, $k\leq n(t)$.
Using this notation, together with Representation \eqv(cor.1), we get that
\eqv(rho.14) is equal to
\begin{align}\Eq(rho.100) 
\frac{1}{\d^2}\E\Biggl[\Bigl\vert \sum_{k=1}^{n(t)}\eee^{(\s+i\rho \t ) x_k(t)-\s m(t)+i\sqrt{1-\rho^2}\t z_k(t)}\1_{\{\widetilde{x}_k(t) <-A\}\cap\{ x_k\in\UU_{t,r,\g} \}  }\Bigr\vert^2\Biggr]
.
\end{align}
Define $\lambda := \s+i\rho \t$. Observe that $|z|^2=z\bar z$, for $z\in \C$.
Hence, the expectation in \eqv(rho.100) is equal to
\begin{align}
\Eq(rho.15)
& \E\Biggl[\sum_{k,l=1}^{n(t)} \eee^{\overline\lambda  x_l(t)+\lambda x_k(t)-2\s m(t)+i\sqrt{1-\rho^2}\t (z_l(t)-z_k(t))}\1_{\forall_{j \in \{l,k\}}\left(\{\widetilde{x}_j(t)<-A\}\cap\{ x_j\in\UU_{t,r,\g} \} \right)}\Biggr]
\\
&
= \E\Bigg[\sum_{k,l=1}^{n(t)} \Big(\eee^{\overline{\lambda} x_l(t)+\lambda x_k(t)-2\s m(t)}\1_{\forall_{j \in \{l,k\}}\left(\{\widetilde{x}_j(t)<-A\}\cap\{ x_j\in\UU_{t,r,\g} \}\right) }\\
&\quad\quad\qquad\times\E\left[\eee^{i\sqrt{1-\rho^2}\t (z_l(t)-z_k(t))} \mid \mathcal{F}^{\mathbb{T}_t} \right] \Big)\Bigg],
\nonumber
\end{align}
where we used that $(z_k(t), k \leq n(t))$ is, conditionally on $\mathbb{T}_t$, independent from $(x_k(t), k \leq n(t))$. Since $(z_k(t), k \leq n(t))$ is
a BBM on the same GW tree as $x$,  \eqv(rho.15) is equal
to
\begin{align}\Eq(rho.101)
\E\left[\sum_{k,l=1}^{n(t)}\eee^{\overline{\lambda} x_l(t)+\lambda x_k(t)-2\s m(t)+(1-\rho^2)\t^2\left(t-d(x_l(t),x_k(t))\right)}\1_{\forall_{j \in \{l,k\}}\{\widetilde{x}_j(t)<-A\}\cap\{ x_j\in\UU_{t,r,\g} \} }\right]
.
\end{align}
We introduce the time of the most recent common ancestor $q_{k,l}=d(x_k(t),x_l(t))$, where $d(\cdot,\cdot)$ is defined in \eqref{eq:intro:gw-overlap}, 
and rewrite \eqv(rho.101) as $\sum_{B>1}T(B,t)$, where
\begin{align}\Eq(rho.16)
T(B,t) :=
\E\left[\sum_{k,l=1}^{n(t)}\eee^{\overline{\lambda}x_l(t)+\lambda x_k(t)-2\s m(t)}\eee^{(1-\rho^2)\t^2\left(t-q_{k,l}\right)}\1_{\UU_{B,q,t}^{l,k}}\right],
\end{align}
and
\begin{align}\Eq(rho.20)
\begin{aligned}
\UU_{B,q,t}^{l,k}& := \cap_{j \in \{l,k\}}  \{\widetilde{x}_j(t)<-A\}\cap \{x_j(s)\leq U_{t,\g}(s), \forall s\in[r,t-r]\}\\
&\quad \cap\{x_j(q_{k,l})-U_{t,\g}(q_{k,l})\in [-B+1,-B]\}.
\end{aligned}
\end{align}
Similar to \eqv(rho.6), we now relax conditions on the path of the particle.
If $q_{k,l}>\frac{3}{4}t$, then we get
\begin{align}\Eq(rho.17)
\UU_{B,q,t}^{l,k}
\subset & \cap_{j \in \{l,k\}}  \{\widetilde{x}_j(t)<-A\}\cap\{x_l(q_{k,l},t)-U_{t,\g}(q_{k,l})\in [-B+1,-B]\}\\
& \cap\{\xi^q_l(s)\leq 8 (s\wedge(q_{k,l}-s))^\g,  \forall s\in[B^{1/\g}\lor r,q_{k,l}-(B^{1/\g}\wedge r)]\} =: \TT^{l,k}_{B,q,t},
\nonumber
\end{align}
where $\xi^q_l(s) := x_l(s,t)-\frac{s}{q}x_l(q_{k,l},t)$ is a Brownian bridge from $0$ to $0$
in time $q_{k,l}$, which is, in particular, independent of $x_l(q_{k,l},t)$. Moreover, for
$q\leq \frac{3}{4}t$, we have
\begin{align}\Eq(rho.50)
\UU_{B,q,t}^{l,k}\subset \cap_{j \in \{l,k\}}  \{\widetilde{x}_j(t)<-A\}\cap\{x_l(q_{k,l},t)-U_{t,\g}(q_{k,l})\in [-B+1,-B]\} =: \SS_{B,q,t}^{l,k}.
\end{align}
Hence, $T(B,t)$ defined in \eqv(rho.16) is bounded from above by
\begin{align}
\Eq(rho.18)
\nonumber
& \E\left[\sum_{k,l=1}^{n(t)}\eee^{\overline{\lambda}x_l(t)+\lambda x_k(t)-2\s m(t)}\eee^{(1-\rho^2)\t^2\left(t-q_{k,l}  \right)}\Big(
\1_{\{ q_{k,l} >\frac{3}{4}t\} \cap \TT^{l,k}_{B,q,t}}
 + \1_{\{ q_{k,l} \leq \frac{3}{4}t \} \cap  \SS_{B,q,t}^{l,k}}\Big)\right] 
\\
\nonumber
& =K\int_0^t \ddd q \ \eee^{2t-q+(1-\rho^2)\t^2\left(t-q\right)} \int_{U_{t,\g}(q)-B}^{U_{t,\g}(q)-B+1} \ddd x
\int_{-\infty}^{m(t)-A-x} \ddd y \int_{-\infty}^{m(t)-A-x} \ddd y' 
\\
&\quad\times
\eee^{\s(2x+y+y'-2m(t))+i\rho\tau(y'-y)}\eee^{-\frac{y^2+y'^2}{2(t-q)}}\sfrac{1}{2\pi(t-q)} \eee^{-\frac{x^2}{2q}}\sfrac{1}{\sqrt{2\pi q}}
\\
\nonumber
& \quad\times\Bigg(\1_{\{q\leq \frac{3}{4}t\}}+\1_{\{q\geq \frac{3}{4}t\}}\P\left\{\xi^q(s)\leq 8(s\wedge(q-s))^\g,  \forall s\in[B^{1/\g}\lor r,q-B^{1/\g}\wedge r]\right\}\Bigg)
,
\end{align}
where $K = \sum_{k=1}^\infty k(k-1) p_k$. It is in \eqv(rho.18) that we need the second
moment assumption on the distribution $(p_k)_{k \geq 0}$,
cf.~Footnote~\ref{note:3}. First, observe that, for $B<t^\g/3$, as in
\eqv(rho.11), the probability in \eqv(rho.18) is bounded from above by
$
\frac{2 (B^{1/\g}\lor r)}{q-2(B^{1/\g}\lor r)}(1+\e)
$.
Observe that $m(t)-A-x\leq m(t)-A-U_{t,\g}(q)+B $. We compute first the
integrals with respect to $y$ and $y'$ in \eqv(rho.18), i.e.,
\begin{align}\Eq(rho.21)
\int_{-\infty}^{\DD_{A,B,q}}\int_{-\infty}^{\DD_{A,B,q}}\eee^{\s(2x+y+y'-2m(t))+i\rho\tau(y'-y)}\eee^{-\frac{y^2+y'^2}{2(t-q)}}\frac{\ddd y \ddd y'}{2\pi(t-q)},
\end{align}
where $\DD_{A,B,q} := m(t)-A-U_{t,\g}(q)+B$.
We make the following change of variables
\begin{align}\Eq(rho.22')
y=w+\lambda (t-q)\quad\mbox{and}\quad y'=w'+\overline{\lambda} (t-q).
\end{align}
Hence, \eqv(rho.21) is equal to
\begin{align}\Eq(rho.22)
\eee^{2\s(x-m(t))+(\s^2-(\rho\t)^2)(t-q)}\int_{-\infty}^{\DD_{A,B,q}-\lambda(t-q)}\int_{-\infty}^{\DD_{A,B,q}-\overline{\lambda}(t-q)}\eee^{-\frac{w^2+w'^2}{2(t-s)}}\sfrac{\ddd w \ddd w'}{2\pi(t-q)}
.
\end{align}
Using \eqv(Gauss.tail), we bound 
\eqv(rho.22) from above by
\begin{multline}
\Eq(rho.23)
\eee^{2\s(x-m(t))+(\s-\t^2)(t-q)}\Big(\1_{\{\DD_{A,B,q}\geq  \s(t-q)\}}\\
+
\exp\left(-\sfrac{\left(\DD_{A,B,q}-\l(t-q)\right)^2+\left(\DD_{A,B,q}-\overline{\l}(t-q)\right)^2}{2(t-q)}\right)
\1_{\{\DD_{A,B,q}\leq \s(t-q)\}}\Big).
\end{multline}
Next we carry out the integration over $x$ in \eqv(rho.18).
Note that 
\begin{align}\Eq(rho.24)
\int_{U_{t,\g}(q)-B}^{U_{t,\g}(q)-B+1}\eee^{2\s x}\eee^{-\frac{x^2}{2q}}\frac{\ddd x}{\sqrt{2\pi q}} 
=
\eee^{2\s^2 q}\int_{U_{t,\g}(q)-B-2\s q}^{U_{t,\g}(q)-B+1-2\s q}\eee^{-\frac{v^2}{2q}}\frac{\ddd v}{\sqrt{2\pi q}},
\end{align}
where we made the change of variables $x=v+2\s q$. Observe that $U_{t,\g}(q)-2\s
q\leq (\sqrt 2-2\s)q<0$, since $\s\geq \frac{1}{\sqrt 2}$. Therefore, using \eqv(Gauss.tail), the right-hand side of 
\eqv(rho.24) is bounded from above by
\begin{align}\Eq(rho.25)
\frac{\sqrt{q}}{2\s q-U_{t,\g}(q)+B}\eee^{2\s^2 q}\eee^{-(U_{t,\g}(q)-B-2\s q)^2/2q} 
.
\end{align}
Using the bounds \eqv(rho.25) and \eqv(rho.23) in \eqv(rho.18), we get that
\eqv(rho.18) is bounded from above by
\begin{align}\Eq(rho.26)
&K\int_0^t 
\frac{\sqrt{q}\eee^{2t-q+2\s^2 q}\eee^{-(U_{t,\g}(q)-B-2\s q)^2/2q}}{2\s q-U_{t,\g}(q)+B}
\eee^{-2\s m(t)+(\s^2-\t^2)(t-q)}\nonumber\\
&\quad \times\Big(\1_{\{\DD_{A,B,q}\geq \sigma (t-q)\}}
+
\eee^{-\frac{\left(\DD_{A,B,q}-\l(t-q)\right)^2+\left(\DD_{A,B,q}-\overline{\l}(t-q)\right)^2}{2(t-q)}}\1_{\{\DD_{A,B,q}\geq \sigma (t-q)\}}\Big)
\nonumber\\
&\quad \times 
\left(\1_{\{q\leq \frac{3}{4}t\}}+\1_{\{q\geq \frac{3}{4}t,\ B<t^\g/3\}}\sfrac{2(B^{1/\g}\lor r)}{q-2(B^{1/\g}\lor r)}(1+\e)\right)
\ddd q
.
\end{align}
Using that $U_{t,\g}(q)-2\s q=(\sqrt 2-2\s) q-\frac{q}{t}\frac{3}{2\sqrt 2}\log
t+(q\wedge (t-q))^\g$, we start to simplify \eqv(rho.26). We
get
\begin{align}\Eq(rho.40)
& \eee^{2t-q}\eee^{2\s^2 q}\eee^{-(U_{t,\g}(q)-B-2\s q)^2/2q}
\eee^{-2\s m(t)+\frac{(\sigma^2 - \tau^2)(t-q)}{2}}\nonumber\\
& \underset{t \uparrow \infty}{\sim} 
\eee^{(t-q)\left((\s-\sqrt 2)^2-\t^2\right)+\left(\frac{3\s}{\sqrt 2}+\frac{(\sqrt 2-2\s)3q}{2\sqrt 2 t}\right)\log t-(\sqrt2 -2\s)(q\wedge (t-q))^\g+(\sqrt 2-2\s)B}
.
\end{align}
Note that by assumption on $\s$ and $\t$ we have $(\s-\sqrt 2)^2-\t^2<0$ and
$\sqrt 2-2\s<0$. Cutting the domain of integration in \eqv(rho.26) into three
parts $q\in[0,t-\log(t)^\a]$, $q\in( t-\log(t)^\a,t-\frac{A}{2}]$ and
$q\in(t-\frac{A}{2},t]$, for some fixed $\a>1$, we get the following three terms
\begin{align}\Eq(rho.51)
K\int_{0}^{t} \dots \ddd q=K \left( \int_{0}^{t-\log(t)^\a} + \int_{t-\log(t)^\a}^{t-\frac{A}{2}} + \int_{t-\frac{A}{2}}^{t} \right) \dots \ddd q =: K\left(\text{(I1)}+\text{(I2)}+\text{(I3)}\right).
\end{align}
We bound (I1) from above by
\begin{align}\Eq(rho.41)
&\int_0^{t-\log(t)^\a}\eee^{(t-q)\left((\s-\sqrt 2)^2-\t^2\right)+\left(\frac{(\sqrt 2-2\s)3q}{2\sqrt 2 t}+\frac{3\s}{\sqrt 2}\right)\log t-(\sqrt2 -2\s)(q\wedge (t-q))^\g+(\sqrt 2-2\s)B}\ddd q (1+o(1)) 
\nonumber\\
&\leq\eee^{(\sqrt 2-2\s)B+\frac{3\s}{\sqrt 2}\log t}\int_0^{t-\log(t)^\a}\eee^{(t-q)\left((\s-\sqrt 2)^2-\t^2\right)-(\sqrt2 -2\s)(q\wedge (t-q))^\g }\ddd q (1+o(1)) 
\nonumber\\
&\leq  \eee^{(\sqrt 2-2\s)B}\eee^{C\log(t)^\a \left((\s-\sqrt 2)^2-\t^2\right)+\frac{3\s}{\sqrt 2}\log t-(\sqrt2 -2\s)\log (t)^{\g\a}},
\quad t \uparrow \infty,
\end{align}
for some constant $C>0$. Hence,
\begin{align}\Eq(rho.sum1)
K\sum_{B>1}\text{(I1)}
\leq K \eee^{C\log(t)^\a \left((\s-\sqrt 2)^2-\t^2\right)+\frac{3\s}{\sqrt 2}\log t-(\sqrt2 -2\s)\log (t)^{\g\a}} \sum_{B>1}\eee^{(\sqrt 2-2\s)B},
\end{align}
since $\sqrt 2-2\s<0$, we have $\sum_{B>1}\eee^{(\sqrt 2-2\s)B}<\infty$. Hence,
we can choose $t_0$ such that, for all $t>t_0$, the r.h.s.\ of \eqv(rho.sum1) less than $\frac{\e}{6}$.
For $q\in( t-\log(t)^\a,t]$, we observe first that
\begin{align}\Eq(rho.42)
\eee^{\left(\frac{(\sqrt 2-2\s)3q}{2\sqrt 2 t}+\frac{3\s}{\sqrt 2}\right)\log t}
\underset{t \uparrow \infty}{\sim}
\eee^{\frac{3}{2}\log t}
,
\end{align} 
and, moreover,
\begin{align}\Eq(rho.43)
\frac{2\sqrt{q}(B^{1/\g}\lor r)}{\left(2\s q-U_{t,\g}(q)+B\right)\left(q-2(B^{1/\g}\lor r)\right)}\leq C' \frac{2(B^{1/\g}\lor r)}{\sqrt{t}(t-2(B^{1/\g}\lor r))},
\end{align}
for some constant $C'>0$. Using \eqv(rho.42) and \eqv(rho.43), we bound (I2)
from above by
\begin{multline}
\Eq(rho.44)
\int_{t-\log(t)^\a}^{t-\frac{A}{2}}\eee^{(t-q)\left((\s-\sqrt 2)^2-\t^2\right) -(\sqrt2 -2\s) (t-q)^\g+(\sqrt 2-2\s)B}C't  \ddd q
\\
\times \left(\sfrac{2(B^{1/\g}\lor r)}{(t-2(B^{1/\g}\lor r))}\1_{\{B<t^\g/3\}}+\1_{\{B\geq t^\g/3\}}\right) (1+o(1))
\\
\leq  C_2\eee^{ \frac{A}{2}((\s-\sqrt 2)^2-\t^2)}\eee^{(\sqrt 2-2\s)B}\left((B^{1/\g}\lor r)\1_{\{B<t^\g/3\}}+t\1_{\{B\geq t^\g/3\}}\right)(1+o(1))
,
\end{multline}
as $t \uparrow \infty$. Using \eqv(rho.44), we get that $K\sum_{B>1}\text{(I2)}$
is bounded from above by
\begin{align}\Eq(rho.sum2)
KC_2\eee^{ \frac{A}{2}((\s-\sqrt 2)^2-\t^2)}\left(\sum_{B=1}^{[t^\g/3]} \eee^{(\sqrt 2-2\s)B}(B^{1/\g}\lor r)
+\sum_{B>[t^\g/3]} \eee^{(\sqrt 2-2\s)B}t \right)(1+o(1))
,
\end{align}
as $t \uparrow \infty$. Again, since $2-2\s<0$, we have $\sum_{B>1}
B^{\frac{1}{\g}}\eee^{(\sqrt 2-2\s)B} <\infty$ and $(\s-\sqrt 2)^2-\t^2<0$.
Hence, there exist $t_1$ and $A_1$ such that, for all $t>t_1$ and all $A>A_1$, we
have that \eqv(rho.sum2)$\leq \frac{\e}{6}$. 
Since $\DD_{A,B,q}-\s(t-q)<0$ for $t-q\leq \frac{A}{\sqrt 2}$ and $B\leq
\frac{A}{2}$, we bound (I3) from above by
\begin{align}\Eq(rho.46)
&\int_{t-\frac{A}{2}}^t\eee^{(t-q)\left((\s-\sqrt 2)^2-\t^2\right)} \eee^{ -(\sqrt2 -2\s) (t-q)^\g+(\sqrt 2-2\s)B}C't\left(\sfrac{2(B^{1/\g}\lor r)}{(t-2(B^{1/\g}\lor r))}\1_{\{B < t^\g/3\}}+\1_{\{B \geq t^\g/3\}}\right)\nonumber\\
&\times  
\Big(\1_{\{B<\frac{A}{2}\}}\eee^{-\frac{ \left( (1-\sqrt 2 \s)A\right)^2 }{(t-q)}}  (1+o(1))+\1_{\{B\geq \frac{A}{2}\}}\Big)\ddd q, 
\quad t \uparrow \infty
.
\end{align}
Using that $(\s-\sqrt 2)^2-\t^2<0$ and $\sqrt2-2\s<0$, we bound \eqref{rho.46}
from above by
\begin{align}\Eq(rho.60)
&\int_{t-\frac{A}{2}}^t \eee^{ -(\sqrt2 -2\s) (\frac{A}{2})^\g+(\sqrt 2-2\s)B}\tilde C \left(\1_{\{B<t/3\}}2(B^{1/\g}\wedge r)+t\1_{\{B\geq t/3\}}\right) \nonumber\\
&\quad \times  
\Big(\1_{\{B<\frac{A}{2}\}}\eee^{-\frac{ \left( (1-\sqrt 2 \s)A\right)^2 }{A/2 }}  (1+o(1))+\1_{\{ B\geq \frac{A}{2}\}}\Big)\ddd q\nonumber\\
&\leq \frac{A}{2}\eee^{ -(\sqrt2 -2\s) (\frac{A}{2})^\g+(\sqrt 2-2\s)B}\tilde C \left(\1_{\{ B<t^\g/3 \}}2(B^{1/\g}\wedge r)+t\1_{\{B\geq t^\g/3\}}\right) \nonumber\\
&\quad \times  
\Big(\1_{\{B<\frac{A}{2}\}}\eee^{-\frac{ \left( (1-\sqrt 2 \s)A\right)^2 }{A/2}}  (1+o(1))+\1_{\{ B\geq \frac{A}{2}\}}\Big),
\quad t \uparrow \infty.
\end{align}
Using \eqv(rho.60), together with the fact that, for all $t>\frac{3A^\g}{2}$, it
holds that $\frac{t^\g}{3}>\frac{A}{2}$, we get that, for all such $t$, the sum
$K\sum_{B>1}\text{(I3)}$ is bounded from above by
\begin{align}\Eq(rho.sum3)
&K\tilde C\frac{A}{2}\eee^{ -(\sqrt2 -2\s) \left(\frac{A}{2}\right)^\g}
\Big(
\sum_{B>1}^{A/2}\eee^{(\sqrt 2-2\s)B}
\eee^{-\frac{ 2\left((1-\sqrt 2 \s)A\right)^2 }{A}}  (B^{1/\g}\lor r)\nonumber\\
&
+
\sum_{B>A/2}^{t^\g/3} \eee^{(\sqrt 2-2\s)B}(B^{1/\g}\lor r)
+ \sum_{B>t^\g/3}t\eee^{(\sqrt 2-2\s)B}\Big)(1+o(1)), \quad t \uparrow \infty.
\end{align}
Hence, there exist $t_2$ and $A_2$ such that for all $t>t_2$ and $A>A_2$ the
term in $\eqv(rho.sum3)$ is not greater than $\frac{\e}{6}$. Now, combining the bounds in
\eqv(rho.sum1), \eqv(rho.sum2) and \eqv(rho.sum3), we get that, for all
$t>\max\{t_0,t_1,t_2\}$ and $A>\max\{A_1,A_2\}$,
$
\sum_{B\geq 1}T(B,t)\leq \frac{\e}{6}+\frac{\e}{6}+\frac{\e}{6}=\frac{\e}{2}
$.
By \eqv(rho.3), this concludes the proof of Proposition~\thv(Prop.nocont).
\end{proof}

\appendix

\section{Martingale convergence}
\label{sec:martingale-convergence}

For $\b=\s+i\t$, set 
$
M_{\b}(t) := \eee^{-t\left(1+\frac{\s^2}{2}-\frac{\t^2}{2}+i\rho\tau\right)}\sum_{k=1}^{n(t)}\eee^{\s x_k(t)+i\t y_k(t)}.
$

\begin{proposition}\TH(Prop.martingale)
For $\b \in \C$ with $|\beta|<1$, $M_{\b}(t)$ is an $L^2$-bounded
martingale with expectation one. In particular, $M_{\b}(t)$ converges to a non-degenerate limit $M_\b$ a.s.\ and in $L^2$
as t tends to infinity.
\end{proposition}

\begin{proof}
Using Representation \eqv(cor.1), one easily verifies that $\E[M_\b(t)]=1$ and
that it is indeed a martingale. It remains to show the $L^2$-boundedness of
$M_\b(t)$. We have
\begin{align}
\Eq(a.2)
\E\left[\vert M_\b(t) \vert^2\right]
=\eee^{-2t\left(1+\frac{\s^2}{2}-\frac{\t^2}{2}\right)}\E\left[ 
\sum_{k,l=1}^{n(t)} \eee^{\s (x_k(t)+x_l(t))+i\t (y_k(t)-y_l(t))}
\right].
\end{align}
Using Representation \eqv(cor.1), we rewrite the right-hand side of \eqv(a.2) as
\begin{align}
\Eq(a.3)
\eee^{-2t\left(1+\frac{\s^2}{2}-\frac{\t^2}{2}\right)}\E\left[ 
\sum_{k,l=1}^{n(t)} \eee^{\bar \l x_l(t)+\l x_k(t))+i\t (1-\rho^2) (z_k(t)-z_l(t))}
\right],
\end{align}
where $\l= \s+i\rho \t$ and $(z_k(t))_{k\leq n(t)}$ are the particles of a BBM
on $\mathbb{T}_t$ that is independent from $X(t)$. By conditioning on
$\mathcal{F}^{\mathbb{T}_t}$ as in \eqv(rho.15), we have that \eqv(a.3) is equal
to
\begin{align}
\Eq(a.4)
 \eee^{-2t\left(1+\frac{\s^2}{2}-\frac{\t^2}{2}\right)}\E\left[ \eee^{-(1-\rho^2)\t^2\left(t-d(x_k(t),x_l(t))\right)}
\sum_{k,l=1}^{n(t)} \eee^{\bar \l x_l(t)+\l x_k(t)}
\right].
\end{align}
Similarly to \eqv(rho.18), the expectation in \eqv(a.4) is equal to
\begin{align}\Eq(a.5)
&K \int_0^t \ddd q \eee^{2t-q-(1-\rho^2)\t^2(t-q)}\int_{-\infty}^{\infty} \frac{\ddd x}{\sqrt{2\pi q}} \int_{-\infty}^{\infty}\frac{\ddd y}{\sqrt{2\pi (t-q)}}\nonumber\\
&\quad\times \int_{-\infty}^{\infty}\frac{\ddd y'}{\sqrt{2\pi (t-q)}} \eee^{2\s x +\s(y+y')+i\t\rho (y-y')}\eee^{-\frac{y^2+y'^2}{2}}\eee^{-x^2/2}.
\end{align}
Computing first the integrals with respect to $y$ and $y'$, we get that \eqv(a.5) is equal to
\begin{align}\Eq(a.6)
&K \int_0^t \ddd q \eee^{2t-q-(1-\rho^2)\t^2(t-q)+(\s^2-\rho^2\t^2)(t-q)}\int_{-\infty}^{\infty} \frac{\ddd x}{\sqrt{2\pi q}}  \eee^{2\s x} \eee^{-x^2/2}\nonumber\\
&= K \int_0^t \ddd q \eee^{2t-q-\t^2(t-q)+\s^2(t-q)} \eee^{2\s^2q}.
\end{align}
Plugging \eqv(a.6) back into \eqv(a.4), we get that \eqv(a.4) is equal to
\begin{align}\Eq(a.7)
\eee^{-2t\left(1+\frac{\s^2}{2}-\frac{\t^2}{2}\right)}K \int_0^t \ddd q \eee^{2t-q-\t^2(t-q)+\s^2(t-q)} \eee^{2\s^2q}
=K \int_0^t \ddd q \eee^{q(\s^2+\t^2-1)}\leq C,
\end{align}
for some constant $C>0$ uniformly in $t$ since $\s^2+\t^2<1$ by assumption.
Hence, $M_\b (t) $ is an $L^2$-bounded martingale with expectation one and
converges as $t\uparrow \infty$ to a non-degenerate limit a.s.~and in $L^2$.

\end{proof}

\ACKNO{We thank Anton~Bovier, Patrik~Ferrari, \mbox{Zakhar}~Kabluchko and the anonymous referee for useful remarks.}

\end{document}